\documentclass[12pt]{amsart}
\usepackage{fullpage,url,amssymb,enumerate,colonequals}

\usepackage{mathrsfs} 
\usepackage{MnSymbol}
\usepackage{extarrows}
\usepackage{lscape}
\usepackage[all,cmtip]{xy}

\usepackage[OT2,T1]{fontenc}

\usepackage{color}

\usepackage[
        colorlinks, citecolor=darkgreen,
        backref,
        pdfauthor={Nuno Freitas}, 
]{hyperref}

\usepackage{comment}

\newcommand{\F}{\mathbb{F}}
\newcommand{\Fbar}{{\overline{\F}}}

\newcommand{\Q}{\mathbb{Q}}

\newcommand{\Z}{\mathbb{Z}}
\newcommand{\Qbar}{{\overline{\Q}}}

\newcommand{\Ebar}{{\overline{E}}}

\newcommand{\rhobar}{{\overline{\rho}}}

\newcommand{\fp}{\mathfrak{p}}

\DeclareMathOperator{\Aut}{Aut}

\DeclareMathOperator{\Det}{Det}

\DeclareMathOperator{\Gal}{Gal}

\newcommand{\GL}{\operatorname{GL}}

\newcommand{\SL}{\operatorname{SL}}

\numberwithin{equation}{section}

\newtheorem{theorem}{Theorem}
\newtheorem{lemma}{Lemma}

\theoremstyle{definition}

\theoremstyle{remark}

\definecolor{darkgreen}{rgb}{0,0.5,0}

\begin{document}

\title[The generalized Fermat equation]%
      {On the Fermat-type equation $x^3 + y^3 = z^p$}

\author{Nuno Freitas}

\date{\today}

\begin{abstract}
  We prove that the Fermat-type equation $x^3 + y^3 = z^p$ has no solutions
  $(a,b,c)$ satisfying $abc \ne 0$ and $\gcd(a,b,c)=1$ when $-3$ is not a 
  square mod~$p$. This improves to approximately $0.844$ the Dirichlet density of the
  set of prime exponents to which the previous equation is known to not have such solutions.
  
  For the proof we   develop a criterion of independent interest to decide if two elliptic curves with certain type of potentially good reduction at 2 have symplectically or anti-symplectically isomorphic $p$-torsion modules.
\end{abstract}

\maketitle

\section{Introduction}

In this paper we consider the Fermat-type equation
\begin{equation} \label{E:gfe}
 x^3 + y^3 = z^p 
\end{equation}
which is a particular case of the {\it Generalized Fermat Equation} (GFE)
\[
  x^p + y^q = z^r, \qquad p, q, r \in \Z_{\geq 2}, \qquad 1/p + 1/q + 1/r < 1.  
\]
Here we are concerned with solutions $(a,b,c)$ which are {\it non-trivial} and {\it primitive}, that is $abc \ne 0$ and $\gcd(a,b,c) = 1$, respectively. To the triple of exponents $(p,q,r)$ we call the {\it signature} of the equation.

The equation \eqref{E:gfe} is one of the few instances of the GFE where there is a known Frey curve
defined over $\Q$ attached to it. The other few signatures with available rational Frey curves are
$(p,p,p)$, $(p,p,2)$, $(p,p,3)$, $(5,5,p)$, $(7,7,p)$, $(2,3,p)$ and $(4,p,4)$ (see \cite{darmon-cr} for their explicit definitions).\footnote{There are also Frey curves attached to signatures of the form $(r,r,p)$ and $(2\ell,2m,p)$ 
but defined over totally real fields (see \cite{F} and \cite{AnniSiksek}).} 
However, only for the signatures $(p,p,p)$, $(4,p,4)$, $(p,p,2)$ and $(p,p,3)$  
the existence of a Frey curve led to a full resolution of the corresponding equation. 
The first due to the groundbreaking work of Wiles \cite{Wiles} and the other three due to work of Darmon \cite{darmon4p4} 
and Darmon-Merel \cite{DM}.
Among the remaining signatures, equation \eqref{E:gfe} is the one where most progress was achieved so far,
due to the work of Kraus \cite{Kraus1998} and Chen--Siksek \cite{ChenSiksek}. 

\begin{theorem}[Kraus 1998] Let $p \geq 17$ be a prime and $(a,b,c)$ be a non-trivial primitive solution to \eqref{E:gfe}. Then $\upsilon_2(a) = 1$, $\upsilon_2(b) = 0$, $\upsilon_2(c) = 0$, and $\upsilon_3(c) \geq 1$.  

Moreover, there are no solutions for exponents $p$ satisfying  $17 \leq p < 10^4$.

\label{T:Kraus}
\end{theorem}

\begin{theorem}[Chen--Siksek 2009] For a set of primes $\mathcal{L}$ with density $0.681$ the equation
\eqref{E:gfe} has no non-trivial primitive solutions. The primes in $\mathcal{L}$ are
determined by explicit congruence conditions, for example $p \equiv 2,3$ mod~$5$. 

Moreover, there are no solutions for exponents $p$ satisfying $3 \leq p \leq 10^7$. 
\end{theorem}

In this work our main goal is to prove the following theorem.

\begin{theorem} Let $p \geq 17$ be a prime satisfying $(-3/p) = -1$, that is $p \equiv 2$ mod~$3$.
Then equation \eqref{E:gfe} has no non-trivial primitive solutions. 

Therefore, equation \eqref{E:gfe} has no non-trivial primitive solutions for a set of prime exponents
with density approximately $0.844$.
\label{T:Main}
\end{theorem}

A crucial tool for the proof is the following criterion to decide whether 
two elliptic curves having certain type of potentially good reduction 
at 2 admit a symplectic or anti-symplectic isomorphism between their $p$-torsion modules
(see beginning of Section~\ref{S:crit} for the definitions).

Write $\Q_2^{un}$ for the maximal unramified extension of $\Q_2$.

\begin{theorem}
  Let $E/\Q_2$ and $E'/\Q_2$ be elliptic curves with potentially good reduction. 
  Write $L =\Q_2^{\text{un}}(E[p])$ and $L' =\Q_2^{\text{un}}(E'[p])$.
  Write $\Delta_m(E)$ and $\Delta_m(E')$ for the minimal discriminant of $E$ and $E'$ respectively.
  Let $I_2 \subset \Gal(\Qbar_2 / \Q_2)$ be the inertia group.
  
  Suppose that $L = L'$ and $\Gal(L/\Q_2^{\text{un}}) \simeq \SL_2(\F_3)$.
  Then, $E[p]$ and $E'[p]$ are isomorphic $I_2$-modules for all prime $p \geq 3$. Moreover,
\begin{enumerate}
 \item if $(2/p)=1$ then $E[p]$ and $E'[p]$ are symplectically isomorphic $I_2$-modules. 
 \item if $(2/p)=-1$ then $E[p]$ and $E'[p]$ are symplectically isomorphic $I_2$-modules
 if and only if $\upsilon_2(\Delta_m(E)) \equiv \upsilon_2(\Delta_m(E')) \pmod{3}$. 
\end{enumerate}
\label{T:maincrit2}
\end{theorem}

This theorem extends the ideas in \cite[Appendice A]{HK2002} and it is proved in Section~\ref{S:crit}; 
in Section~\ref{S:Main} we use it to establish Theorem~\ref{T:Main}.  

In \cite{FNS23n} we develop further symplecticity criteria and apply them to 
the Generalized Fermat Equation $x^2 + y^3 = z^p$.
  
\subsection*{Idea behind the proof} Our proof of Theorem~\ref{T:Main} builds on Kraus' modular argument \cite{Kraus1998}. 
Indeed, for $p \geq 17$ he attaches to a putative non-trivial primitive solution $(a,b,c)$ of 
\eqref{E:gfe} a Frey elliptic curve 
\[
 E_{a,b} : Y^2 = X^3 + 3abX + b^3 - a^3, \qquad \Delta(E_{a,b}) = -2^4 \cdot 3^3 \cdot c^{2p}
\]
and shows that its mod~$p$ Galois representation $\rhobar_{E_{a,b},p}$ is mostly independent of $(a,b,c)$.
By the now classic modularity, irreducibility and level lowering results over $\Q$ it follows that 
$\rhobar_{E_{a,b},p}$ is isomorphic to $\rhobar_{f,\fp}$ the mod~$\fp$ representation attached to a
rational newform $f$ in a finite list. Finally, among all the possibilities for $f$ Kraus obtains a contradiction 
except for the newform corresponding to the rational elliptic curve with Cremona label $72a1$. 

In particular, following the ideas in \cite{PSS}, Kraus' work implies that the solution~$(a,b,c)$ gives rise to a
rational point on one of the modular curves $X_{72a1}^+(p)$ or~$X_{72a1}^-(p)$;
these curves respectively parameterize elliptic curves with $p$-torsion modules 
symplectically or anti-symplectically isomorphic to the $p$-torsion module of $72a1$.
By applying Theorem~\ref{T:maincrit2} and \cite[Proposition 2]{KO} we will show that 
there are no $2$-adic points in $X_{72a1}^-(p)$ and 3-adic points in $X_{72a1}^{-(-3/p)}(p)$
arising from relevant solutions of \eqref{E:gfe}. In particular, when $(-3/p) = -1$ 
this implies there are no relevant points on $X_{72a1}^{\pm}(p)(\Q)$.

\subsection*{Acknowledgments} I would like to thank Benjamin Matschke, Bartosz Naskr{\k e}cki and Michael Stoll for  
helpful discussions. I also thank Alain Kraus for his comments.

\section{Proof of Theorem~\ref{T:Main}}
\label{S:Main}

Let $(a,b,c)$ be a non-trivial primitive solution to $x^3 + y^3 = z^p$. From Theorem~\ref{T:Kraus} we know that
$\upsilon_2(a) = 1$, $\upsilon_2(b) = 0$, $\upsilon_2(c) = 0$ and $\upsilon_3(c) \geq 1$
and we can attach to it the Frey curve
\[
 E_{a,b} : Y^2 = X^3 + 3abX + b^3 - a^3.
\]
A closer look into Kraus' proof shows also that
the mod~$p$ Galois representation of $E_{a,b}$ has to satisfy 
$\rhobar_{E_{a,b}, p} \sim \rhobar_{W^\prime, p}$,
where $W^\prime$ is the elliptic curve with Cremona label $72a1$. 
Moreover, this possibility is the unique obstruction to conclude that \eqref{E:gfe} has 
no non-trivial primitive solutions. We shall show that $\rhobar_{E_{a,b}, p} \not\sim \rhobar_{W^\prime, p}$ when $(-3/p)=-1$.

Note that
$W^\prime$ has potentially multiplicative reduction at 3, which becomes multiplicative
after twisting by $-3$. Write $E$ and $W$ for the quadratic twists by $-3$ of $E_{a,b}$ and $W^\prime$, 
respectively. Thus we have
\begin{equation}
  \rhobar_{E, p} \sim \rhobar_{W, p},
 \label{E:iso}
\end{equation}
where $W$ has Cremona reference $24a4$ with $j$-invariant $j_W = 2048/3$ and minimal model
\[
 W :  Y^2 = X^3 - X^2 + X.
\]
Since $\upsilon_2(j_W) = 11$ the curve $W$ has potentially good reduction at 2 and it gets good reduction over $L=\Q_2^{un}(W[p])$. The curve $W$ also satisfies 
\[
\upsilon_2(\Delta_m(W)) = 4 \qquad \text{ and } \qquad \upsilon_2(c_4(W)) = 5, 
\]
hence $\Gal(L/\Q_2^{un}) \simeq \SL_2(\F_3)$ by \cite{Kraus1990}. From \eqref{E:iso} the same must be true for $E$, therefore we are under the hypothesis of Theorem~\ref{T:maincrit2}.

From part (2.2) in the proof of \cite[Lemma~4.1]{Kraus1998} we have that 
$E_{a,b}$ is minimal at $2$ and satisfies $\upsilon_2(\Delta_m(E_{a,b})) = 4$.
Hence the same is true for the quadratic twist $E=-3E_{a,b}$ and we have 
$\upsilon_2(\Delta_m(E)) \equiv \upsilon_2(\Delta_m(W)) \pmod{3}$.
We conclude from Theorem~\ref{T:maincrit2} that $E[p]$ and $W[p]$ are
symplectically isomorphic $I_2$-modules for all $p \geq 3$. Since $\rhobar_{W,p}(I_2)$ is
non-abelian, $E[p]$ and $W[p]$ are also symplectically isomorphic as $G_\Q$-modules
by \cite[Lemma~A.4]{HK2002}; moreover, they cannot be simultaneously symplectic and
anti-symplectic isomorphic by Lemma~\ref{L:sympcriteria}.

From \cite[Proposition~2]{KO} applied with the multiplicative prime $\ell=3$
it follows that $E[p]$ and $W[p]$ are symplectically isomorphic if and only if
$\upsilon_3(\Delta_m(W))$ and $\upsilon_3(\Delta_m(E))$
differ multiplicatively by a square modulo $p$. We now compute these quantities. 

One easily checks that $\upsilon_3(\Delta_m(W)) = 1$.

From part (3.1) in the proof of \cite[Lemma~4.1]{Kraus1998} we see that
\[
 \upsilon_3(c_4(E_{a,b})) = 2, \quad \upsilon_3(c_6(E_{a,b})) = 3, \quad \upsilon_3(\Delta(E_{a,b})) = 3 + 2p\upsilon_3(c).
\]
Therefore, the twisted curve $E=-3E_{a,b}$ satisfies
\[
 \upsilon_3(c_4(E)) = 4, \quad \upsilon_3(c_6(E)) = 6, \quad \upsilon_3(\Delta(E)) = 9 + 2p\upsilon_3(c).
\]
Since $\upsilon_3(c) \geq 1$ it follows from Table~II in \cite{pap} that the equation for $E$ is not minimal.
After a change of variables we obtain 
\[
\upsilon_3(c_4) = 0,\qquad \upsilon_3(c_6) = 0, \qquad \upsilon_3(\Delta_m(E)) = -3 + 2p\upsilon_3(c)  
\]
and the model gets multiplicative reduction. Therefore, $E[p]$ and $W[p]$ are symplectically isomorphic if and only if
\[
 1=\upsilon_3(\Delta_m(W)) \equiv u^2 \upsilon_3(\Delta_m(E)) = u^2 (-3 + 2p\upsilon_3(c)) \pmod{p}
\]
which is equivalent to $(-3/p) = 1$. The result follows.

The statement about the density follows by the same computations as in \cite[Section~10]{ChenSiksek} but now 
we also take into account the congruence $p \equiv 2$ mod~$3$.

\section{Symplectic isomorphisms of the $p$-torsion of elliptic curves}
\label{S:crit}

Let $p$ be a prime.
Let $K$ be a field of characteristic zero or a finite field of characteristic $\neq p$
with an algebraic closure $\overline{K}$.
Fix $\zeta_p \in \overline{K}$ a primitive $p$-th root of unity.
For $E$ an elliptic curve defined over~$K$ we write $E[p]$ for its $p$-torsion 
$G_K$-module, $\rhobar_{E,p} \colon G_K \to \Aut(E[p])$ for the corresponding
Galois representation and $e_{E,p}$ for the Weil pairing on~$E[p]$.
We will call an $\F_p$-basis $(P,Q)$ of $E[p]$ \emph{symplectic} if 
$e_{E,p}(P,Q) = \zeta_p$.

Now let $E /K$ and $E'/K$ be two elliptic curves and
$\phi \colon E[p] \to E'[p]$ be an isomorphism of $G_K$-modules.
Then there is an element $r(\phi) \in \F_p^\times$ such that
\[ e_{E',p}(\phi(P), \phi(Q)) = e_{E,p}(P, Q)^{r(\phi)} \quad \text{for all $P, Q \in E[p]$.} \]
Note that for any $a \in \F_p^\times$ we have $r(a\phi) = a^2 r(\phi)$.
We say that $\phi$ is a \textit{symplectic isomorphism} if $r(\phi) = 1$ or,
more generally, $r(\phi)$ is a square in~$\F_p^\times$.
Fix a nonsquare~$r_p \in \F_p^\times$.
We say that $\phi$ is a \textit{anti-symplectic isomorphism} if $r(\phi) = r_p$
or, more generally, $r(\phi)$ is a nonsquare in~$\F_p^\times$.
Finally, we say that $E[p]$ and~$E'[p]$ are
\emph{symplectically isomorphic} (or \emph{anti-symplectically isomorphic}),
if there exists a symplectic (or anti-symplectic) isomorphism of~$G_K$-modules between them.
Note that it is possible that $E[p]$ and~$E'[p]$ are both symplectically
and anti-symplectically isomorphic; this will be the case if and only if
$E[p]$ admits an anti-symplectic automorphism.

We will need the following criterion.

\begin{lemma} \label{L:sympcriteria}
  Let $E$ and $E'$ be two elliptic curves defined over a field $K$
  with isomorphic $p$-torsion.
  Fix symplectic bases for $E[p]$ and $E'[p]$. Let $\phi \colon E[p] \to E'[p]$ be
  an isomorphism of $G_K$-modules and write $M_\phi$ for the matrix representing~$\phi$ with
  respect to the fixed bases.

  Then $\phi$ is a symplectic isomorphism if and only if $\det(M_\phi)$ is a square mod~$p$;
  otherwise $\phi$ is anti-symplectic.

  Moreover, if $\rhobar_{E,p}(G_K)$ is a non-abelian subgroup of~$\GL_2(\F_p)$,
  then $E[p]$ and~$E'[p]$ cannot be simultaneously symplectically and anti-symplectically isomorphic.
\end{lemma}

\begin{proof}
  Let $P,Q \in E[p]$ and $P',Q' \in E'[p]$ be symplectic bases. We have that
  \[ e_{E',p}(\phi(P),\phi(Q)) = e_{E',p}(P',Q')^{\det(M_\phi)}
                               = {\zeta_p}^{\det(M_\phi)}
                               = e_{E,p}(P,Q)^{\det(M_\phi)},
  \]
  so $r(\phi) = \det(M_\phi)$. This implies the first assertion.

  We now prove the last statement.
  Let $\beta \colon E[p] \rightarrow E'[p]$ be another isomorphism of $G_K$-modules.
  Then $\beta^{-1} \phi = \lambda$ is in the centralizer of~$\rhobar_{E,p}(G_K)$.
  Since $\rhobar_{E,p}(G_K)$ is non-abelian,
  $\lambda$ is represented by a scalar matrix (see \cite[Lemme~A.3]{HK2002}).
  Therefore $\det(M_\beta)$ and~$\det(M_\phi)$ are in the same square class mod~$p$.
\end{proof}

\medskip 

We now introduce notation from \cite[Section 2]{ST1968} and \cite[Appendice A]{HK2002}.
Let $p \neq \ell$ be primes such that $p \geq 3$.
For an elliptic curve $E / \Q_\ell$ with potentially good reduction 
write $L=\Q_\ell^{\text{un}}(E[p])$. Write also $I=\Gal(L/\Q_\ell^{\text{un}})$.
Write $\Ebar$ for the elliptic curve over $\Fbar_\ell$ obtained by 
reduction of a minimal model of $E/L$ and $\varphi : E[p] \rightarrow \overline{E}[p]$ for the 
reduction morphism which is a symplectic isomorphism of $G_L$-modules.
Let $\Aut(\Ebar)$ be the automorphism group of $\Ebar$ over $\Fbar_\ell$ 
and write $\psi : \Aut(\Ebar) \rightarrow \GL(\Ebar[p])$ for the natural
injective morphism. The action of $I$ on $L$ induces an injective morphism 
$\gamma_E : I \rightarrow \Aut(\overline{E})$. Moreover, for $\sigma \in I$
we have
\begin{equation}
 \varphi \circ \rhobar_{E,p}(\sigma) = \psi(\gamma_E(\sigma)) \circ \varphi.
 \label{eqn:gamma}
\end{equation}

The following group theoretical lemma is proved in Section~\ref{S:group}.
For convenience we state it here since it plays a crucial r\^ole in the proof of 
Theorem~\ref{T:maincrit2}. 

\begin{lemma} \label{L:normalizer2}
  Let $p \geq 3$ and $G = \GL_2(\F_p)$. Let $H \subset \SL_2(\F_p) \subset G$ be 
  a subgroup isomorphic to~$\SL_2(\F_3)$.
  Then the group~$\Aut(H)$ of automorphisms of~$H$ satisfies
  \[ N_G(H)/C(G) \simeq  \Aut(H) \simeq S_4, \]
  where $N_G(H)$ denotes the normalizer of $H$ in $G$ and $C(G)$ the center of $G$. Moreover, 
  \begin{enumerate}[\upshape(a)]
    \item if $(2/p) = 1$, then all the matrices in~$N_G(H)$ have square determinant;
    \item if $(2/p) = -1$, then the matrices in~$N_G(H)$ with square determinant
          correspond to the subgroup of~$\Aut(H)$ isomorphic to~$A_4$. 
  \end{enumerate}
\end{lemma}

\begin{proof}[Proof of Theorem~\ref{T:maincrit2}] 
Let $E, E'$ be elliptic curves as in the statement.
Note that $L=\Q_2^{\text{un}}(E[p])$ is the smallest extension of $\Q_2^{\text{un}}$ where $E$ obtains good reduction
and the reduction map $\varphi$ is an isomorphism between the $\F_p$-vector spaces $E[p](L)$ and $\Ebar[p](\Fbar_2)$.
By hypothesis $E'$ also has good reduction over $L$ and the same is true for $\varphi'$. Applying 
equation \eqref{eqn:gamma} to both $E$ and $E'$ we see that $E[p]$ and $E'[p]$ are isomorphic $I_2$-modules if we show that
$\psi \circ \gamma_E$ and $\psi \circ \gamma_{E'}$ are isomorphic as representations into $\GL(\Ebar[p])$ and $\GL(\overline{E'}[p])$,
respectively. 

We have that $j(\Ebar) = j(\overline{E'}) = 0$ (see the proof of \cite[Thereom~3.2]{DD2015}) thus $E$ and $E'$ are isomorphic
over $\Fbar_\ell$. So we can fix minimal models of $E/L$ and $E'/L$ both reducing to the same $\Ebar$. 
Write $H:= \Aut(\Ebar)$ and note that $H \simeq \SL_2(\F_3)$ (see \cite[Thm.III.10.1]{SilvermanI}). 
Therefore $\psi(\gamma_E(I)) = \psi(\gamma_{E'}(I)) = \psi(H) \subset \SL(\Ebar[p]) \subset \GL(\Ebar[p])$ and there must be an automorphism $\alpha \in \Aut(\psi(H))$ such that
$\psi(\gamma_E) = \alpha \circ \psi(\gamma_{E'})$. The first statement of Lemma~\ref{L:normalizer2} shows there is 
$g \in \GL(\Ebar[p])$ such that $\alpha(x) = gxg^{-1}$ for all $x \in \psi(H)$; thus $\psi \circ \gamma_E$ and $\psi \circ \gamma_{E'}$ are isomorphic representations. 

Fix a symplectic basis of $\Ebar[p]$ identifying $\GL(\overline{E}[p])$ with $\GL_2(\F_p)$. 
Let $M_g$ denote the matrix representing $g$ and observe that $M_g \in N_{\GL_2(\F_p)}(\psi(H))$.
Lift the fixed basis to bases of $E[p]$ and $E'[p]$ via the corresponding reduction maps $\varphi$ and $\varphi'$. The lifted bases are symplectic. Write $M_\varphi$ and $M_{\varphi'}$ for the matrices representing $\varphi$ and $\varphi'$ on these bases, respectively. 

Set $M:=M_{\varphi'}^{-1} M_g M_\varphi$. From \eqref{eqn:gamma} it follows that $\rhobar_{E,p}(\sigma) = M^{-1}\rhobar_{E',p}(\sigma)M$ for all $\sigma \in I$. Moreover, $M$ represents some $I_2$-modules isomorphism $\phi : E[p] \rightarrow E'[p]$ and from Lemma~\ref{L:sympcriteria} we have that $E[p]$ and $E'[p]$ are symplectically isomorphic if and only if $\det(M)$ is a square mod~$p$. Since $\varphi$ and $\varphi'$ are symplectic isomorphisms of $G_L$-modules
Lemma~\ref{L:sympcriteria} implies the determinants of $M_\varphi$ and $M_{\varphi'}$ are squares mod~$p$. Therefore
$E[p]$ and $E'[p]$ are symplectically isomorphic if and only if $\det(M_g)$ is a square mod~$p$.

Part (1) now follows from Lemma~\ref{L:normalizer2} (a). 

We now prove (2). From Lemma~\ref{L:normalizer2} (b)
we see that $E[p]$ and $E'[p]$ are symplectically isomorphic if and only if $\alpha$ is an automorphism in $A_4 \subset \Aut(\psi(H)) \simeq S_4$. Note that these are precisely the inner automorphisms. 
For each $p$ the map $\alpha_{p} := \psi^{-1}\circ\alpha\circ\psi$ defines an automorphism of $\gamma_E(I)=H=\Aut(\Ebar)$ satisfying $\alpha_{p} \circ \gamma_{E'} = \gamma_E$. Since $\gamma_E$, $\gamma_{E'}$ 
are surjective and independent of $p$ it follows that $\alpha_{p}$ is the same for all $p$.
Since $\alpha$ and $\alpha_p$ are simultaneously inner or not it follows this property is independent
of the prime $p$ satisfying $(2/p) = -1$. This shows that $E[p]$ and $E'[p]$ are symplectically
isomorphic $I_2$-modules if and only if $E[\ell]$ and $E'[\ell]$ are symplectically
isomorphic $I_2$-modules for one (hence all) $\ell$ satisfying $(2/\ell) = -1$. 

We are left to show that symplecticity 
is equivalent to $\upsilon_2(\Delta_m(E)) \equiv \upsilon_2(\Delta_m(E')) \pmod{3}$. 

Since $(2/3) = -1$ from the observation above we can work with $p=3$. 

Fix $\omega \in \overline{\F}_2$ a primitive cubic root of unity.
Let $L_3 \subset L$ be an extension of $\Q_2^{un}$ of degree 8. Hence $L/L_3$ is cyclic
of degree 3 and we write $\sigma$ for a generator of $G = \Gal(L/L_3) \subset I$.
Thus $\gamma_E(G)$ and $\gamma_{E'}(G)$ are order 3 subgroups of $\Aut(\overline{E})$. 

Recall that $\psi : \Aut(\overline{E}) \rightarrow \GL(\overline{E}[3])$ is the natural injective morphism. 
After fixing a symplectic basis for $\overline{E}[3]$, conjugation by an element of $\SL_2(\F_3)$ (which preserves the property of a basis of $\overline{E}[3]$ being symplectic) allows to assume that $\psi(\gamma_E(G))$ is the group generated by
$\begin{pmatrix} 1 & 1 \\ 0 & 1 \end{pmatrix}$. In particular, $E$ has a 3-torsion point defined over $L_3$.

By doing the same for $E'$ we obtain
$\psi(\gamma_E(\sigma)) = M_g \psi(\gamma_{E'}(\sigma)) M_g^{-1}$, where $M_g$ belongs to the normalizer $N = N_{\GL_2(\F_3)}(\psi(\gamma_E(G)))$. Since the centralizer of $\psi(\gamma_E(\sigma))$ consists precisely of the elements of $N$ with square determinant it follows
that
\[
 \gamma_E(\sigma) = \gamma_{E'}(\sigma) \Leftrightarrow E[3] \simeq E'[3] \text{ symplectically}.
\]
We can further assume that the residual curve $\Ebar$ is of the following form
\[
\overline{E}  \; : \; y^2 + a_3 y = x^3 + a_4 x + a_6, \quad a_i \in \overline{\F}_2,\quad a_3 \ne 0.
\]

For such a model the elements of order 3 in $\Aut(\overline{E})$ are the linear transformations $ T(u) : (x,y) \mapsto (u^2 x, u^3 y)$, where $u = \omega^k$ for $k=0,1,2$. Since $E$ has a 3-torsion point defined over $L_3$, the same argument leading to equation (17) in \cite{HK2002} applies (possibly after replacing $\sigma$ by $\sigma^2$). Thus $\gamma_E(\sigma) = T(\omega^{\upsilon(\Delta_m(E))})$.
By doing the same for $E'$ we get $\gamma_{E'}(\sigma) = T(\omega^{\upsilon(\Delta_m(E'))})$ and the result follows.
\end{proof}

\subsection{A lemma in group theory} \strut
\label{S:group}
Write $S_n$ and $A_n$ for the symmetric and alternating group on $n$ elements, respectively.
We write $C(G)$ for the center of a group $G$.
If $H$ is a subgroup of $G$, then we write $N_G(H)$ for its normalizer
and $C_G(H)$ for its centralizer in~$G$.

\begin{lemma} 
  Let $p \geq 3$ and $G = \GL_2(\F_p)$. Let $H \subset \SL_2(\F_p) \subset G$ be 
  a subgroup isomorphic to~$\SL_2(\F_3)$.
  Then the group~$\Aut(H)$ of automorphisms of~$H$ satisfies
  \[ N_G(H)/C(G) \simeq  \Aut(H) \simeq S_4. \]
  Moreover, 
  \begin{enumerate}[\upshape(a)]
    \item if $(2/p) = 1$, then all the matrices in~$N_G(H)$ have square determinant;
    \item if $(2/p) = -1$, then the matrices in~$N_G(H)$ with square determinant
          correspond to the subgroup of~$\Aut(H)$ isomorphic to~$A_4$. 
  \end{enumerate}
\end{lemma}

\begin{proof} We can write $H$ as $H = \langle i, j, k, u \rangle$ where
\begin{enumerate}
 \item $H_8 = \langle i,j,k \rangle$ is a subgroup isomorphic to the quaternion group; there is no 
 other subgroup of $H$ with order 8, hence $H_8$ is normal in $H$;
 \item $u$ has order 3 and satisies $uiu^{-1}=j$, $uju^{-1}=k$, $uku^{-1}=i$.
\end{enumerate}

Let $\alpha, \beta \in \F_p^\times$ satisfy $\alpha^2 + \beta^2 = -1$
and consider the following elements of $\SL_2(\F_p) \subset G$
  \[ g_1 = \begin{pmatrix} 0 & -1 \\ 1 & 0 \end{pmatrix}, \quad
     g_2 = \begin{pmatrix} \alpha & \beta \\ \beta & -\alpha \end{pmatrix}, \quad 
     g_3 = \frac{1}{2}\begin{pmatrix} \alpha + \beta - 1 & \beta - \alpha - 1 \\ \beta - \alpha + 1& -\alpha-\beta-1 \end{pmatrix}.
  \]

It is known that $H_8$ can be conjugated by an elment $g \in G$ into $\langle g_1, g_2 \rangle $. 
Moreover, we have
$gHg^{-1}  = \langle g_1, g_2, -g_1 g_2, u_g \rangle$ where
$gig^{-1} = g_1$, $gjg^{-1} = g_2$, $gkg^{-1} = -g_1 g_2$, $u_g = gu g^{-1}$.   
One checks that the action by conjugation of $u_g$ and $g_3$ on $\langle g_1, g_2 \rangle$
is equal, therefore $u_g = g_3 \lambda $ for some scalar matrix $\lambda$. 
Since $u_g \in \SL_2(\F_p)$ by taking determinants we see that $\lambda = \pm 1$; 
$\lambda = -1$ is impossible due to order considerations, thus $u_g = g_3$. 
We have shown that we can suppose the generators of $H$ are 
$i = g_1$, $j=g_2$, $k=-g_1 g_2$ and $u = g_3$.

  From \cite[Lemma~A.3]{HK2002} we have $C_G(H) = C(G)$.
  Now the action by conjugation induces a canonical group homomorphism
  $N_G(H) \to \Aut(H)$ with kernel $C_G(H)$, leading to an injection
  $N_G(H)/C(G) \to \Aut(H)$. To see that this map is also surjective (and hence
  an isomorphism), note that $N_G(H)$ contains the matrices
  \[ n_1 = \begin{pmatrix} 1 & -1 \\ 1 & 1 \end{pmatrix} \quad\text{and}\quad
     n_2 = \begin{pmatrix} \alpha & \beta-1 \\ \beta+1 & -\alpha \end{pmatrix}
  \]
  and that the subgroup of~$N_G(H)/C(G)$ generated by the images of~$H$ and of
  these matrices has order~$24$. Since it can be easily checked that $\Aut(\SL_2(\F_3)) \simeq S_4$,
  the first claim follows.

  Note that $A_4$ is the unique subgroup of~$S_4$ of index~2.
  The determinant induces a homomorphism $S_4 \simeq N_G(H)/C(G) \to \F_p^\times/\F_p^{\times 2}$
  whose kernel is either $S_4$ or~$A_4$.
  Since $H \subset \SL_2(\F_p)$ and all matrices in $C(G)$ have square determinant,
  it remains to compute $\det(n_1)$ and~$\det(n_2)$. But $\det(n_1) = 2$ and
  $\Det(n_2) = -\alpha^2 - (\beta - 1)(\beta + 1) = -\alpha^2 - \beta^2 + 1 = 2$
  as well. 
\end{proof}

\end{document}